\documentclass[preprint,12pt]{elsarticle}

\usepackage{amssymb}
\usepackage{amsmath}
\usepackage{microtype}
\journal{arXiv}
\DeclareMathOperator{\cl}{cl}

\begin{document}

\begin{frontmatter}



\title{Uniformly Star Superparacompact Subsets and Spaces} 


\author{Argha Ghosh} 

\affiliation{organization={Department of Mathematics, Manipal Institute of Technology, Manipal Academy of Higher Education},
            addressline={Manipal}, 
            postcode={576104}, 
            state={Karnataka},
            country={India}}

\begin{abstract}
Uniformly star superparacompactness, which is a topological property between compactness and completeness, can be characterized using finite-component covers and a measure of strong local compactness \cite{Adhikary2024, Das2021}. Using these finite-component covers and the associated functional, we introduce and investigate a variational notion of uniformly star superparacompact subsets in metric spaces in the spirit of studies on uniformly paracompact subset and UC-subset. We show that the collection of all such subsets forms a bornology with a closed base, which is contained in the bornology of uniformly paracompact subsets. Conditions under which these two bornologies coincide are specified. Furthermore, we provide several new characterizations of uniformly star superparacompact metric spaces—also known as cofinally Bourbaki-quasi complete spaces in terms of some geometric functionals. As a consequence, we establish new relationships among metric spaces that lie between compactness and completeness.
\end{abstract}






\end{frontmatter}



\section{Introduction}
This study is focused on the two topics—uniformly star superparacompact subsets of metric spaces and uniformly star superparacompact metric spaces. Uniformly star superparacompact metric spaces, which is also known as cofinally Bourbaki quasi-complete (cBq-complete), was introduced and studied recently in both uniform and metric settings in \cite{Adhikary2024, Das2021}. They characterize confinally Boubaki quasi-complete metric spaces in terms of two functionals $f_c$ and $f_p$, uniformly  strong local compactness in metric setting and in terms of finite-component open covers in uniform settings. 
Uniformly star superparacompactness like uniform paracompactness (also know as cofinally completeness in metric setting) is also a topological property between compactness and completeness, and over the past two decades, completeness-like properties that lie in between compactness and completeness have been the subject of significant research by numerous authors, including G. Beer, M.I. Garrido, A.S. Meroño, S. Kundu, M. Aggarwal, L. Gupta and N. Adhikary, among others (see \cite{Aggarwal2017, Beer2008, BeerBook, Das2024, Garrido2014, Garrido2016, Hohti1981} and the references therein). 

The following hierarchy of implications among various completeness-like properties in metric spaces is now well established:

\begin{align*}
\text{compact} \Rightarrow \text{cBq-complete} \Rightarrow \text{cofinally Bourbaki-complete} \\
\Rightarrow \text{cofinally complete} \Rightarrow \text{complete},
\end{align*}
\begin{align*}
\text{compact} \Rightarrow \text{cBq-complete} \Rightarrow \text{Bourbaki quasi-complete} \\
\Rightarrow \text{Bourbaki-complete} \Rightarrow \text{complete}.
\end{align*}

Now we draw attention to some notable resemblances between cofinally Bourbaki quasi-complete metric spaces, cofinally complete metric spaces and UC-spaces.

The space \(\langle X, d\rangle\) is a cofinally Boubaki quasi-complete if and only if every sequence \(\langle x_n \rangle\) in \(X\) satisfying \(\lim_{n\to\infty} f_c(x_n) = 0\) clusters, where the functional $f_c : X \to [0, \infty)$ is defined by
\[
f_c(x) = 
\begin{cases}
\sup \left\{ \varepsilon > 0 : S^\infty_d(x, \varepsilon) \text{ is compact} \right\}, & \exists \varepsilon > 0 \ni S^\infty_d(x, \varepsilon) \text{ is compact}, \\
0, & \text{otherwise}.
\end{cases}
\]
where $S^\infty_d(x, \varepsilon)$ denotes $\varepsilon$-chainable component of $x$ in $X$ (see \cite{Das2024}).

Similarly,\(\langle X, d\rangle\)is cofinally complete if and only if every sequence \(\langle x_n \rangle\) satisfying \(\lim_{n \to \infty} \nu(x_n) = 0\) clusters, where
\[
\nu(x) = 
\begin{cases}
\sup \{\varepsilon > 0 : \mathrm{cl}(S_d(x,\varepsilon)) \text{ is compact} \}, & \text{if } x \text{ has a compact neighborhood}, \\
0, & \text{otherwise}.
\end{cases}
\]
where $S_d(x, \varepsilon)$ denotes $\varepsilon$-neighborhood $x$ in $X$ (see \cite{Beer2008}).

Also, the metric space \(\langle X, d\rangle\) is a UC-space if and only if every sequence $\langle x_n\rangle$ in $X$ satisfying
\[
\lim_{n \to \infty} I(x_n) = 0
\]
clusters, where the isolation function $I : X \to [0, \infty)$ is defined by
\[
I(x) = d\left(x, X \setminus \{x\}\right),
\]
which measures the isolation of $x$ in the space (see \cite{Atsuji1958, BeerBook, Kundu2006}).

A metric space \(\langle X, d\rangle\) is cofinally Boubaki quasi-complete if and only if:
\begin{enumerate}
  \item its set of points with no compact chainable component, denoted \(\operatorname{nslc}(X) := \{x \in X : f_c(x) = 0\}\)=$\operatorname{Ker}f_c$, is compact, and
  \item for all \(\delta > 0\), the set \(\{x \in X : d(x, \operatorname{nslc}(X)) > \delta\}\) is strongly uniformly locally compact (see \cite{Das2024}).
\end{enumerate}

Similarly, \(\langle X, d\rangle\) is cofinally complete if and only if:
\begin{enumerate}
  \item its set of points with no compact neighborhood, denoted \(\operatorname{nlc}(X) := \{x \in X : \nu(x) = 0\}\)=$\operatorname{Ker}\nu$, is compact, and
  \item for all \(\delta > 0\), the set \(\{x \in X : d(x, \operatorname{nlc}(X)) > \delta\}\) is uniformly locally compact (see \cite{Beer2008, Hohti1981}).
\end{enumerate}

Similarly, \(\langle X, d\rangle\) is a UC-space if and only if:
\begin{enumerate}
    \item The set of limit points $X'=\operatorname{Ker} I$ is compact, and
    \item For all $\delta > 0$, the set $\{x \in X : d(x, X') > \delta\}$ is uniformly isolated (see \cite{Atsuji1958, BeerBook, Kundu2006}).
\end{enumerate}

 Both the classes are characterized by Cantor-type theorems: \(\langle X, d\rangle\) is a UC-space (respectively, cofinally complete) if and only if every decreasing sequence \(\langle A_n \rangle\) of nonempty closed subsets of \(X\) along which the set-functional \(\sup \{I(a) : a \in A_n\}\) (respectively, \(\sup \{\nu(a) : a \in A_n\}\)) tends to zero has nonempty intersection \cite{Beer1986}. It is immediate consequence of \cite[Theorem 28.20]{BeerBook} that \(\langle X, d\rangle\) is a cofinally Bourbaki quasi-complete if and only if every decreasing sequence \(\langle A_n \rangle\) of nonempty closed subsets of \(X\) along which the set-functional \(\sup \{f_c(a) : a \in A_n\}\) tends to zero has nonempty intersection.

 In the literature there are several other characterization of UC-spaces and cofinally complete metric spaces exist, for instance see \cite{Aggarwal2016, Aggarwal2017, BeerBook, Beer2010, Garrido2016, Gupta2021, Kundu2006} and references therein. In view of the known parallels between UC spaces, cofinally complete spaces and cofinally Boubaki quasi-complete spaces, it is natural to expect that the class of cofinally Boubaki quasi-complete spaces also has interesting characteristic properties. One of our goals is to further investigate in the class of cofinally Bourkai quasi-complete metric spaces. We note that cofinally Bourbaki quasi-complete spaces lie between UC-spaces and cofinally complete metric spaces. In Section 4, we give several new characterizations of cofinally Bourbaki quasi-complete spaces and we introduce two new geometric functionals to give Kuratowski-type theorems for cofinally Boubaki quasi-complete spaces (for instance, see Theorem \ref{BC-complete_to_USS} and Theorem \ref{BCQ-complete_to_USS} ).

Strongly uniform continuity is a variational alternative to uniform continuity was studied for functions $f : \langle X, d\rangle \to \langle Y, \rho \rangle $ restricted to nonempty subsets of $X$ in \cite{Beer2009}. It is well-know that strong uniform continuity of continuous functions may occur on a bornology that strictly larger than the family $\mathcal{K}(X)$ of compact subsets \cite{Beer2009, Beer2010}. For instance, let $X = (-\infty, 0] \cup \mathbb{N}$ with the usual metric. Then every continuous function $f : X \to \mathbb{R}$ is strongly uniformly continuous on any subset $E = A \cup B$ where $\cl {A}$ is compact and $B \subseteq \mathbb{N}$. In fact that the largest bornology on which each continuous function is strongly uniformly continuous is the bornology of UC-subsets. 

\begin{definition}\cite{Beer2009}
A nonempty subset $A$ of a metric space \(\langle X, d\rangle\) is called a \emph{UC-subset} if whenever $\langle a_n \rangle$ is a sequence in $A$ with $\lim_{n \to \infty} I(a_n) = 0$, then $\langle a_n \rangle$ has a cluster point in $X$.
\end{definition}

Considering the well-known connections between UC-spaces and cofinally complete spaces, Beer et al. in \cite{Beer2012}, introduced the following class of cofinally complete subsets: 

\begin{definition}\cite{Beer2012}
A nonempty subset $A$ of a metric space \(\langle X, d\rangle\) is called a \emph{cofinally complete subset} or \emph{uniformly paracompact subset}or a \emph{CC-subset} if whenever $\langle a_n \rangle$ is a sequence in $A$ with $\lim_{n \to \infty} \nu(a_n) = 0$, then $\langle a_n \rangle$ has a cluster point in $X$.
\end{definition}

In light of the established analogies among UC-spaces, cofinally complete spaces, and cofinally Bourbaki quasi-complete metric spaces, it is natural to anticipate that the class of uniformly star superparacompact subsets—introduced from a different perspective and formally motivated in Section 3—also exhibits significant and distinctive structural properties.

\begin{definition}
A nonempty subset $A$ of a metric space \(\langle X, d\rangle\) is called a uniformly star superparacompact subset if whenever $\langle a_n \rangle$ is a sequence in $A$ with $\lim_{n \to \infty} f_c(a_n) = 0$, then $\langle a_n \rangle$ has a cluster point in $X$.
\end{definition}

This, too, constitutes a variational notion-a subset $A$ does not become a uniformly star superparacompact subset simply by being
a uniformly star superparacompact space in its own right. One must consider the boundary points of the set that lie outside the set itself, inside the metric space in which the set is embedded. For example, the positive integers $\mathbb{N}$, consisting of points that are uniformly isolated, is evidently a uniformly star superparacompact when equipped with the usual metric since $\operatorname{Ker}f_c=\emptyset$. But $\mathbb{N}$ is not a uniformly star superparacompact subset of $\mathbb{R}$, because relative to $\mathbb{R}$, the sequence $1, 2, 3, \ldots$ has no cluster point, while for each $n \in \mathbb{N}$, $f_c(n) = 0$.

 One of the aims of this paper is to explore this newly introduced class of subsets. Our study uncovers not only the anticipated parallels between UC-subsets, CC-subsets and uniformly star superparacompact subsets but also yields novel characterizations of UC-subset, CC-subsets and cofinally Bourbaki quasi-complete metric spaces. 

 \section{Preliminaries}
This section provides the essential definitions required for the subsequent sections. Any additional notions not covered here may be found in \cite{Adhikary2024, BeerBook, Das2021, Das2024}.
All metric spaces will be assumed to contain at least two points. If \( A \) is a subset of \(\langle X, d\rangle\), we denote its closure and set of limit points by \( \cl(A) \) and \( A' \), respectively. For a metric space \(\langle X, d\rangle\), denote by \( S_d(x,\varepsilon) \)), the open ball with center \( x \in X \) and radius \( \varepsilon > 0 \), and for any subset \( A \) of \( X \) and \( \varepsilon > 0 \), we will denote the \( \varepsilon \)-enlargement of \( A \) by
\[
A^\varepsilon = \bigcup \{ S_d(a,\varepsilon): a \in A \} = \{ y : d(y, A) < \varepsilon \}.
\]
Furthermore, the \( \varepsilon \)-chainable component of \( x \in X \) is defined by
\[
S^{\infty}_d(x,\varepsilon) = \bigcup_{n \in \mathbb{N}} S^n_d(x,\varepsilon),
\]
where \( S^1_d(x,\varepsilon) = S_d(x,\varepsilon) \) and for every \( n \geq 2 \),
\[
S^n_d(x,\varepsilon) = (S^{n-1}_d(a,\varepsilon))^\varepsilon.
\]

Let \(\langle X, d\rangle\) be a metric space and \( \varepsilon > 0 \) be given. Then an ordered set of points \( \{x_0, x_1, \dots, x_n\} \) in \( X \) satisfying \( d(x_{i-1}, x_i) < \varepsilon \), where \( i = 1, 2, \dots, n \), is said to be an \( \varepsilon \)-chain of length \( n \) from \( x_0 \) to \( x_n \). 
Note that \( y \in S^n_d(x,\varepsilon) \) if and only if \( x \) and \( y \) can be joined by an \( \varepsilon \)-chain of length \( n \).

\begin{definition}\cite{Atsuji1958, BeerBook}
\begin{enumerate}
    \item A metric space \(\langle X, d\rangle\) is called \( \varepsilon \)-chainable if any two points of \( X \) can be joined by an \( \varepsilon \)-chain, whereas \( X \) is called chainable if \( X \) is \( \varepsilon \)-chainable for every \( \varepsilon > 0 \).
    \item A subset \( A \) of a metric space \(\langle X, d\rangle\) is said to be Bourbaki bounded (also known as finitely chainable) if for every \( \varepsilon > 0 \), there exist \( m \in \mathbb{N} \) and a finite collection of points \( p_1, p_2, \dots, p_k \in X \) such that 
    \[
    A \subseteq \bigcup_{i=1}^{k} S^n_d(p_i,\varepsilon).
    \]
\end{enumerate}
\end{definition}

\begin{definition}\cite{Adhikary2024, Das2021}
Let \(\langle X, d\rangle\) be a metric space. A subset \( A \subseteq X \) is said to be a qC-precompact subset of \( X \) (or sometimes simply called qC-precompact in \( X \)) if for every \( \varepsilon > 0 \), there exists a finite collection of points \( p_1, p_2, \dots, p_k \in X \) such that
\[
A \subseteq \bigcup_{i=1}^{k} S^{\infty}_d(p_i,\varepsilon).
\]
\end{definition}

\begin{definition}\cite{Adhikary2024, Das2021, Das2024, Garrido2014} Let \(\langle X, d\rangle\) be a metric space. 
\begin{enumerate}
    \item A sequence \( \langle x_n \rangle \) is said to be pseudo Bourbaki-Cauchy in \( X \) if for every \( \varepsilon > 0 \) and $n\in\mathbb N$, there exists\( n_0 \in \mathbb{N} \) such that $x_j$ and $x_k$ can be joined by an $\varepsilon$-chain for some $j>k\geq n_0$.
    \item A sequence \( \langle x_n \rangle \) is said to be Bourbaki-Cauchy in \( X \) if for every \( \varepsilon > 0 \), there exist \( m \in \mathbb{N} \) and \( n_0 \in \mathbb{N} \) such that for some \( p \in X \), we have \( x_n \in S^{m}_d(p,\varepsilon) \) for every \( n \geq n_0 \).
    \item A sequence \( \langle x_n \rangle \) is said to be Bourbaki quasi-Cauchy in \( X \) if for every \( \varepsilon > 0 \), there exists \( n_0 \in \mathbb{N} \) such that for some \( p \in X \), we have \( x_n \in S^{\infty}_d(p,\varepsilon) \) for every \( n \geq n_0 \).
    \item  A sequence \( \langle x_n \rangle \) is said to be cofinally Bourbaki quasi-Cauchy in \( X \) if for every \( \varepsilon > 0 \), there exists there exists an infinite subset \( N_\varepsilon \subseteq \mathbb{N} \) such that for some \( p \in X \), we have \( x_n \in S^{\infty}_d(p,\varepsilon) \) for every \( n \in N_\varepsilon \). 
\end{enumerate}

\end{definition}

\begin{definition}\cite{Adhikary2024, Das2021, Das2024, Garrido2014} Let \(\langle X, d\rangle\)be a metric space. 
\begin{enumerate}
    \item $X$ is said to be Bourbaki-complete if every Bourbaki-Cauchy sequence in \( X \) has a cluster point.
    \item $X$ is said to be Bourbaki quasi-complete if every Bourbaki quasi-Cauchy sequence in \( X \) has a cluster point.
    \item $X$ is said to be cofinally Bourbaki quasi-complete (cBq-complete) if every cofinally Bourbaki quasi-Cauchy sequence in \( X \) has a cluster point.
\end{enumerate}

\end{definition}

\begin{definition}\cite{Das2024}
A metric space \(\langle X, d\rangle\)is said to be strongly locally compact if \( f_c(x) > 0 \) for all \( x \in X \). It is said to be strongly uniformly locally compact if \(\inf \{ f_c(x) : x \in X \} > 0.\)
\end{definition}

Let \( V \) be a family of subsets of a set \( X \). We say that \( V \) is a cover of \( A \subseteq X \) provided that \( A \subseteq \bigcup V \). A second family of subsets \( \mathcal{U} \) is said to refine \( V \) if for every \( U \in \mathcal{U} \), there exists \( V \in V \) such that \( U \subseteq V \). Each metric space is, of course, paracompact: every open cover has a locally finite open refinement. A metric space is called uniformly paracompact \cite{Rice1977} if for each open cover \( V \), there exist an open refinement \( \mathcal{U} \) and \( \delta > 0 \) such that for each \( x \in X \), the open ball \( S_d(x,\delta) \) intersects only finitely many members of \( \mathcal{U} \).

\section{Uniformly Star Pracompact Subsets}
\label{sec1}
In this section we will study the family of those subsets \( A \) of a metric space \( \langle X, d \rangle \) that are \emph{uniformly star superparacompact}: for each open cover \( \mathcal{V} \) of \( X \), there exists \( \mu > 0 \) and an open cover \( \mathcal{U} \) refining \( \mathcal{V} \) such that for each \( a \in A \), \( S_d^{\infty}(a, \mu) \) intersects at most finitely many members of \( \mathcal{U} \).

\begin{theorem}\label{coverconditions}
Let \( A \) be a nonempty subset of a metric space \( \langle X, d \rangle \). Then the following conditions are equivalent:
\begin{enumerate}
    \item \( A \) is uniformly star superparacompact;
    \item for each open cover \( \mathcal{V} \) of \( X \), there exists \( \mu > 0 \) such that whenever \( E \subset S^{\infty}_d(a, \mu)  \) for some $a\in A$, there exist \( \{V_1, V_2, \ldots, V_n\} \subseteq \mathcal{V} \) with \( E \subseteq \bigcup_{j=1}^{n} V_j \);
    \item whenever \( \{V_j : j \in \Lambda\} \) is an open cover of \( X \) directed by inclusion, there exists \( \mu > 0 \) such that for all \( a \in A \), there exists \( j \in \Lambda \) with \( S^{\infty}_d(a, \mu) \subseteq V_j \).
    \item whenever \( \mathcal{U} \) is a locally finite open cover of \( X \), there exists \( \mu > 0 \) such that for each \( a \in A \), \( S^{\infty}_d(a, \mu) \) intersects at most finitely many members of \( \mathcal{U} \);
\end{enumerate} 
\end{theorem}

\begin{proof}
 $(1)\Rightarrow (2)$ By $(1)$, we can choose \( \mu > 0 \) and \( \mathcal{U} \) an open cover refining \( \mathcal{V} \) such that for each \( a \in A \), \( S^{\infty}_d(a, \mu) \) intersects at most finitely many elements of \( \mathcal{U} \). Let \( E \) satisfy $E\subseteq S^{\infty}_d(a_0, \mu)$ for some $a_0\in A$. Let \( U_1, U_2, \ldots, U_n \) be those members of the cover that \( S^{\infty}_d(a_0, \mu) \) hits. Choose for \( j = 1, 2, \ldots, n \) elements \( V_j \in \mathcal{V} \) with \( U_j \subseteq V_j \). As \( \mathcal{U} \) is a cover of \( X \),
\[
E \subseteq S^{\infty}_d(a_0, \mu) \subseteq \bigcup_{j=1}^{n} U_j \subseteq \bigcup_{j=1}^{n} V_j
\]
as required.  

$(2) \Rightarrow  (3)$
Let \( \{V_j : j \in \Lambda\} \) be an open cover of \( X \) directed by inclusion, i.e.,  whenever \(j_1, j_2\in \Lambda \), there exists \( j_3 \in \Lambda \) with \( V_{j_1} \cup V_{j_2} \subseteq V_{j_3} \). Choose by $(2)$ \( \mu > 0 \) such that if $E\subseteq S^{\infty}_d(a, \mu)$ for some $a\in A$, then \( E \) is contained in a finite union of members of the cover. Let $a\in A$. Then by the assumption, there exists \( \{j_1, j_2, j_3, \ldots, j_k\} \subseteq \Lambda \) with
\[
S^{\infty}_d(a, \mu) \subseteq \bigcup_{l=1}^{k} V_{j_l}.
\]
But there exists \( j_{n+1} \in \Lambda \) such that
\[
\bigcup_{l=1}^{n} V_{j_l} \subseteq V_{j_{n+1}}.
\]
Hence condition $(3)$ holds.

$(3)\Rightarrow  (4)$ Let \( \mathcal{U} \) be a locally finite open cover of \( X \). For each \( x \in X \), choose \( \delta_x > 0 \) such that \( S_d(x, \delta_x) \) intersects only finitely many members of \( \mathcal{U} \). Let \( \mathcal{W} \) be the cover of \( X \) consisting of all finite unions of members of \( \{ S_d(x, \delta_x) : x \in X \} \). Obviously, \( \mathcal{W} \) is an open cover of \(X\) directed by inclusion. Thus, by $(3)$ there exists \( \mu > 0 \) such that for all \( a \in A \), \( S^{\infty}_d(a, \mu) \) is contained in a finite union of members of \( \{ S_d(x, \delta_x) : x \in X \} \). Consequently, \( S^{\infty}_d(a, \mu) \) hits only finitely many members of \( \mathcal{U} \).

$(4)\Rightarrow  (1)$ It follows from the fact that every metric space is paracompact: every open cover of $X$ has a locally finite open refinement. 

\end{proof}

\begin{corollary}
Let \( \langle X, d \rangle \) be a metric space. The family of uniformly star superparacompact subsets of \( X \) forms a bornology with closed base.  
\end{corollary}
\begin{proof}
 By condition $(2)$ of Theorem \ref{coverconditions} it is obvious that the family of uniformly star superparacompact subsets forms a bornology of $X$. Let \( A \) be a uniformly star superparacompact nonempty subset of $X$. Let \( \mathcal{V} \) be an open cover of \( X \). By $(2)$ of Theorem \ref{coverconditions}, choose \( \mu > 0 \) such that whenever \( E \subset S^{\infty}_d(a, \mu)  \) for some $a\in A$, there exist \( \{V_1, V_2, \ldots, V_n\} \subseteq \mathcal{V} \) with \( E \subseteq \bigcup_{j=1}^{n} V_j \). If \( E \subset S^{\infty}_d(a, \mu)  \) for some $a\in \cl (A)$. Then choose $b\in A$ such that $d(a, b)<\mu$. Then \(S^{\infty}_d(a, \mu)=S^{\infty}_d(b, \mu)\). Thus \(E\) has a finite subcover from \( \mathcal{V} \). 
 \end{proof}

\begin{definition}
 Let \( \langle X, d \rangle \) and \( \langle Y, \rho \rangle \) be metric spaces and let \( f \in Y^X \). We say \( f \) is \emph{uniformly locally component bounded} on \( A \subseteq X \) if there exists \( \delta > 0 \) such that for all \( a \in A \), \( f(S^{\infty}_d(a, \delta)) \) is a metrically bounded subset of \( Y \).   
\end{definition} 

\vspace{0.5em}

\begin{theorem}\label{functional_char}
Let \( A \) be a non-empty subset of a metric space \( \langle X, d \rangle \). Then the following conditions are equivalent:
\begin{enumerate}
     \item \( A \) is uniformly star superparacompact;
    
    \item If \( f : \langle X, d \rangle \to \langle Y, \rho \rangle \) is continuous, then \( f \) is uniformly locally component bounded on \( A \);
    
    \item If \( f \in C(X, \mathbb{R}) \), then \( f \) is uniformly locally component bounded on \( A \).
    \item Each sequence \( \langle a_n \rangle \) in \( A \) such that \( \lim_{n \to \infty} f_c(a_n) = 0 \) clusters;
    
    \item \( \cl{A} \cap \operatorname{nslc}(X) \) is compact, and for all \( \delta > 0 \), there exists \( \mu > 0 \) such that whenever \( a \in A \setminus S_d(\cl{A} \cap \operatorname{nslc}(X), \delta) \), we have \( f_c(a) > \mu \);

\end{enumerate}
 
\end{theorem}
\begin{proof}
$(1)\Rightarrow (2)$ Fix \( y_1 \in Y \). Then \( \{ f^{-1}(S_\rho(y_1, n)) : n \in \mathbb{N} \} \) is an open cover of \( X \). Thus, by assumption, there exists \( \mu > 0 \) such that for each \( a \in A \), \( S^{\infty}_d(a, \mu) \) is contained in a finite union of these preimage sets. Hence \( f(S^{\infty}_d(a, \mu)) \) is a metrically bounded subset of \( Y \).   

$(2)\Rightarrow  (3)$ This is obvious.

$(3)\Rightarrow  (4)$ Suppose to the contrary that there exists a sequence \( \langle a_n \rangle \) in \( A \) for which \( \lim_{n \to \infty} f_c(a_n) = 0 \) but it does not have a cluster point. Without loss of generality, we may assume that its terms are distinct. Let \( B = \{ a_n : n \in \mathbb{N} \} \). Then $B$ is closed and discrete.

Now suppose \( B \) is a qC-precompact set. Then by passing to a subsequence, we may assume that \( \langle a_n \rangle \) is a Bourbaki quasi-Cauchy sequence. Let \( f \in C(X, \mathbb{R}) \) map each \( a_n \) to \( n \). Then \( f \) fails to be uniformly locally component bounded on $B$, which is a contradiction.

Now suppose \( B \) is not a qC-precompact set. Then by Theorem 2.7 \cite{Adhikary2024, Das2021}, $B$ contains an infinite uniformly chain discrete subset. Therefore, by passing to a subsequence, we can find \( \delta > 0 \) such that \( \{ S^{\infty}_d(a_n, \delta) : n \in \mathbb{N} \} \) is a pairwise disjoint family of $\delta$-chainable components. Since \( \lim_{n \to \infty} f_c(a_n) = 0 \), by passing to a subsequence we can assume $f_c(a_n)<\frac{\delta}{n}$. Then for each \( n \in \mathbb{N} \), \(S^{\infty}_d(a_n, \frac{\delta}{n})\) is not compact. For each \( n \), let \( f_n : S^{\infty}_d(a_n, \frac{\delta}{n}) \to \mathbb{R} \) be continuous and unbounded. Now \( \bigcup_{n=1}^{\infty} S^{\infty}_d(a_n, \frac{\delta}{n}) \) is closed. Indeed, if $a$ is a limit point of the set and $S_d(a,\delta)$ intersects \(S^{\infty}_d(a_n, \frac{\delta}{n})\), then $S_d(a,\delta)$ intersects $ S^{\infty}_d(a_n, \delta)$. Hence $a\in S^{\infty}_d(a_n, \delta)$. Since \( \{ S^{\infty}_d(a_n, \delta) : n \in \mathbb{N} \} \) is pairwise disjoint, $a\in S^{\infty}_d(a_n, \frac{\delta}{n})$. Thus there exists \( f \in C(X, \mathbb{R}) \) that extends each \( f_n \). This function is also not uniformly locally component bounded on the set of terms.

$(4)\Rightarrow  (5)$ This is an immediate consequence of \cite[Theorem 28.10.]{BeerBook}.

 $(5) \Rightarrow  (1) $ Suppose \( A \) satisfies condition $(5)$. It is sufficient to show that \( A \) satisfies condition $(2)$ of Theorem \ref{coverconditions}. Let \( \mathcal{V} \) be an open cover of \( X \). If \( \operatorname{nslc}(X) \cap \cl{A} = \emptyset \), then \(
\inf\{ f_c(a) : a \in A \} > 0. \)
Take \( \mu > 0 \) such that for all \( a \in A \), \( S^{\infty}_d(a, \mu) \) is compact. Then whenever $E\subseteq S^{\infty}_d(a, \mu)$ for some $a\in A$, then \( E \) lies in the union of finitely many members of the cover.

Otherwise, \( \operatorname{nslc}(X) \cap \cl{A} \) is nonempty and compact. Let \( \mathcal{V}_1 \) be a finite subfamily of \( \mathcal{V} \) such that
\[
\operatorname{nslc}(X) \cap \cl{A} \subseteq \bigcup \mathcal{V}_1.
\]
By compactness, there exists \( \varepsilon > 0 \) such that
\[
S_d(\text{nlc}(X) \cap \cl{A}, \varepsilon) \subseteq \bigcup \mathcal{V}_1.
\]
Clearly, for all \( x \in S_d(\text{nlc}(X) \cap \cl{A}, \varepsilon/2) \), we have
\[
S_d\left(x, \frac{\varepsilon}{2}\right) \subseteq \bigcup \mathcal{V}_1.
\]
By the given condition, there exists \( \delta > 0 \) such that for all \( a \in A \),
\[
d(a, \operatorname{nslc}(X) \cap \cl{A}) > \frac{\varepsilon}{3} \quad \Rightarrow \quad f_c(a) > \delta.
\]
Set
\(
\mu = \min\left\{ \delta, \frac{\varepsilon}{3} \right\}
\). Fix $a\in A$. Then $S_d(a, \mu)\subseteq S_d(a, \frac{\varepsilon}{3})$. Thus $f_c(a)>\delta\geq\mu$. Hence $S_d^{\infty}(a, \mu)$ is compact, and so is contained in a finite union of members of \( \mathcal{V} \).

\end{proof}

\begin{corollary}

Each Boubaki quasi-Cauchy sequence in a nonempty uniformly star superparacompact subset \( A \) of a metric space \( \langle X, d \rangle \) clusters. As a result, \( \cl{A} \) as a metric subspace of \( \langle X, d \rangle \) is Bourbaki quasi-complete.   
\end{corollary}
\begin{proof}
Suppose \( \langle a_n \rangle \) is a Boubaki quasi-Cauchy sequence in \( A \) that does not cluster. Then by passing to a subsequence, we can assume $f_c(a_n)\leq \frac 1 n$. Thus \( \lim_{n \to \infty} f_c(a_n) = 0 \), which contradicts condition $(4)$ of Theorem \ref{functional_char}.  
\end{proof}

Another immediate consequence of condition (4) of Theorem \ref{functional_char} is the following corollary:
\begin{corollary}
 If $X$ is a finite-dimensional normed linear space, then $A$ is compact if and only if $A$ is uniformly star superparacompact.   
\end{corollary}

\begin{corollary}
 Let \( A \) be a subset of a metric space \( \langle X, d \rangle \). Suppose each point of \( \cl{A} \) has a strongly locally compact chainable component in \( X \). Then \( A \) is uniformly star superparacompact if and only if there exists \( \mu > 0 \) such that for each \( a \in A \), the set $S_d^{\infty}(a, \mu)$ is compact.   
\end{corollary}
\begin{proof}
 This is immediate from condition $(5)$ of Theorem \ref{functional_char} since \( \cl{A} \cap \operatorname{nslc}(X) = \emptyset \).   
\end{proof}

Note that if \( f_c(x_0) = \infty \) for some \( x_0 \in X \), then \( f_c(x) = \infty \) for all \( x \in X \), and if \( f_c \) is finite-valued function, then \( f_c \) is 1-Lipschitz function on $X$. Consequently, $f_c$ is strongly uniformly continuous on each non-empty uniformly star superparacompact subset of $X$. 

\begin{corollary}
 Let \( \langle X, d \rangle \) be a metric space. Then the bornology of uniformly star superparacompact subsets $\mathcal B^\text{uss}$ of \( X \) is shield from closed sets and contains UC-subsets of $X$. 
\end{corollary}

\begin{proof}
If $f_c(x)=\infty$ for some $x\in X$, then $\mathcal B^\text{uss}=2^X$ and $X$ will be the required superset for each member of the boronology. If $f_c$ is finite-valued, then $f_c$ is strongly uniformly continuous on $\mathcal B^\text{uss}$. Hence by (3) of \cite[Proposition 28.13]{BeerBook}, we conclude that $\mathcal B^\text{uss}$ is shield from closed sets. That each UC-subset is uniformly star superparacompact follows from the fact that $I(x)\leq f_c(x)$ and in view of condition (4) of Theorem \ref{functional_char}.
\end{proof}


\begin{remark}
In view of condition $(2)$ of Theorem \ref{coverconditions} it follows from condition $(2)$ of Theorem 29.1. \cite{BeerBook} that each uniformly star superparacompact subset is uniformly paracompact. Therefore, we have the following proposition.
 \end{remark}

\begin{proposition}
 Let \( A \) be a nonempty subset of a metric space \( \langle X, d \rangle \). The following conditions are equivalent:
\begin{enumerate}
 \item \( \langle A, d \rangle \) is a compact metric space;
 \item \( A \) is a UC-subset of each metric space in which it is isometrically embedded.
 \item \( A \) is a uniformly star superparacompact subset of each metric space in which it is isometrically embedded.
\item \( A \) is a uniformly paracompact subset of each metric space in which it is isometrically embedded;

\end{enumerate}
   
\end{proposition}

\begin{theorem}
Let \( A \) be a nonempty subset of a metric space \( \langle X, d \rangle \). The following conditions are equivalent:
\begin{enumerate}
    \item \( A \in \mathcal B^\text{uss} \);
    \item whenever \( \rho \) is a metric equivalent to \( d \) on \( X \), there exists \( \mu > 0 \) such that for all \( a \in A \), \( S^{\infty}_d(a, \mu) \) is \( \rho \)-bounded.
\end{enumerate}
\end{theorem}

\begin{proof}
(1) \( \Rightarrow \) (2) Let \( A \in \mathcal B^\text{uss} \). Then by condition $(2)$ of Theorem \ref{functional_char}, the identity map \( I_X : \langle X, d \rangle \to \langle X, \rho \rangle \) is uniformly locally component bounded on \( A \), and we have condition $(2)$.

(2) \( \Rightarrow \) (1) Suppose to the contrary \( A \notin \mathcal B^\text{uss} \); then by condition $(3)$ of Theorem \ref{functional_char} there exists \( f \in C(X, \mathbb{R}) \) such that \( f \) fails to be uniformly locally component bounded on \( A \). Consider the equivalent metric \( \rho \) on \( X \) defined by
\[
\rho(x, w) = d(x, w) + |f(x) - f(w)|.
\]
For each \( \delta > 0 \), there exists \( a \in A \) such that \( f(S^{\infty}_d(a, \delta)) \) is an unbounded subset of \( \mathbb{R} \), and so \( S^{\infty}_d(a, \delta) \) is not $\rho$-bounded, which is a contradiction. This completes the proof of (2) \( \Rightarrow \) (1).
\end{proof}

We now give a simple characterization of qC-precompact subset of a metric space in terms of cofinally Bourbaki quasi-Cauchy sequence.

\begin{proposition}
Let \( A \) be a nonempty subset of a metric space \( \langle X, d \rangle \). The following conditions are equivalent:
\begin{enumerate}
    \item \( A \) is qC-precompact;
    \item each sequence in \( A \) is cofinally Bourbaki quasi-Cauchy;
    \item each sequence in \( A \) is Boubaki pseudo-Cauchy.
\end{enumerate}
\end{proposition}

\begin{proof}
(1) \( \Rightarrow \) (2) Suppose \( A \) is qC-precompact and \( \langle a_n \rangle \) is a sequence in \( A \). Let \( \varepsilon > 0 \), and choose a finite subset \( \{p_1,\ldots,p_k\} \) of \( X \) such that \( A \subseteq \bigcup_{i=1}^k S^{\infty}_d(p_i, \varepsilon) \). Then there exists an infinite subset \( N_\varepsilon \) of \( \mathbb{N} \) and $j\in \{1,\ldots k\} $ with $a_n\in S^{\infty}_d(p_j, \varepsilon)$ for each \( n \in N_\varepsilon \). Consequently, \( \langle a_n \rangle \) is cofinally Boubaki quasi-Cauchy. 

(2) \( \Rightarrow \) (3) This is obvious.

(3) \( \Rightarrow \) (1) Suppose to the contrary if (1) fails, then for some \( \varepsilon_0 > 0 \) there exists no finite subset \( F \) of \( A \) for which \( A \subseteq S^{\infty}_d(F, \varepsilon_0) \). Let \( a_1 \in A \). Choose \( a_2 \in A \) such that $a_1$ and $a_2$ cannot be joined by an $\varepsilon_0$-chain. We can inductively find \( a_3, a_4, \ldots \) in \( A \) with \( a_{n+1} \notin \bigcup_{i=1}^n S^{\infty}_d(a_i, \varepsilon) \) for each \( n \in \mathbb{N} \), and the sequence \( \langle a_n \rangle \) so constructed is not Boubaki pseudo-Cauchy, which is a contradiction.
\end{proof}

When \( A \) is the entire metric space, we have seen that clustering of each cofinally quasi-Cauchy sequence is necessary and sufficient for uniform star superparacompactness of the space \cite[Theorem 4.2]{Das2024}. 

\begin{proposition}\label{first_ncsc}
 For a non-empty subset $A$ of a metric space \( \langle X, d \rangle \). The following conditions are equivalent:
 \begin{enumerate}
    \item \( A \) is uniformly star superparacompact;
    \item whenever \( \langle x_n \rangle \) is a sequence in \( X \) with both \(\lim_{n \to \infty} d(x_n, A) = 0\) and
    \\ \( \lim_{n \to \infty} f_c(x_n) = 0\), then \( \langle x_n \rangle \) clusters.   
\end{enumerate}    
\end{proposition}
\begin{proof}
 (1) \( \Rightarrow \) (2) Firstly, suppose $X$ is compact, then it is obvious. Suppose $X$ is not compact, then $f_c$ is finite-valued. Thus $f_c$ is strongly uniformly continuous on $A$. Since \(\lim_{n \to \infty} d(x_n, A) = 0\), then there exists a sequence \( \langle a_n \rangle \) in $A$ such that $d(x_n, a_n)\leq \frac 1 n.$ Then by strongly uniform continuity of $f_c$ on $A$, we have \( \lim_{n \to \infty} f_c(a_n) = 0\). Hence \( \langle a_n \rangle \) clusters, and so \( \langle x_n \rangle \) clusters.
 
(2) \( \Rightarrow \) (1) This is obvious in view of condition (4) of Theorem \ref{functional_char}.

\end{proof}

\begin{definition}
Let $A$ be a non-empty subset of a metric space \( \langle X, d \rangle \). A sequence  \( \langle x_n \rangle \) is said to be \textit{cofinally Bourbaki quasi-Cauchy with respect to $A$} if for each $\varepsilon>0$, there exist an infinite subset $N_{\varepsilon}$ of $\mathbb N$ and $a\in A$ such that $x_n\in S_d^{\infty}(a, \varepsilon)$ for all $n\in N_{\varepsilon}$.  
\end{definition}

\begin{proposition}
 For a non-empty subset $A$ of a metric space \( \langle X, d \rangle \). If \( A \) is uniformly star superparacompact, then whenever \( \langle x_n \rangle \) is a sequence in \( X \) with both \(\lim_{n \to \infty} d(x_n, A) = 0\) and cofinally Bourbaki quasi-Cauchy sequence with respect to $A$ in $X$ clusters.     
\end{proposition}
 
\begin{proof}
 For each $n\in \mathbb N$, there exist an infinite subset $N_{\frac 1 n}$ of $\mathbb N$ and $a_n\in A$ such that $x^n_j\in S_d^{\infty}(a_n, \frac 1 n)$ for all $j\in N_{\frac 1 n}$. If for some $n\in\mathbb N$ $\langle x^n_j \rangle_{j\in N_{\frac 1 n}}$ clusters, then \( \langle x_k \rangle \) clusters. Otherwise, $f_c(a_n)\leq\frac 1 n$. Easily, for the subsequence \( \langle x_{k_n} \rangle \) of \( \langle x_k \rangle \) defined by $x_{k_n}=x^n_n$ we have $f_c(x_{k_n})\leq \frac 1 n.$ As a result,  \( \lim_{n \to \infty} f_c(x_{k_n}) = 0\). Hence by Proposition \ref{first_ncsc}, \( \langle x_{k_n} \rangle \) clusters. Thus, \( \langle x_n \rangle \) clusters. 
\end{proof}



\section{Uniformly star superparacompact spaces}

We now consolidate the results from the previous section to provide a characterization of uniformly star superparacompact metric spaces, also known as cofinally Bourbaki quasi-complete metric spaces. The equivalence of all conditions, except for the final two, follows directly from the analysis conducted in the preceding section. The last two conditions are presented in \cite[Theorem 28.20]{BeerBook}. The equivalence of conditions (2), (8), and (9) have been proved in \cite[Theorem 4.2]{Das2024} also.

\begin{theorem}\label{equivalentconditionsforspace}
Let \(\langle X, d \rangle\) be a metric space. Then the following conditions are equivalent:
\begin{enumerate}
    \item \(X\) is a uniformly star superparacompact space;
    \item each cofinally Bourbaki quasi-Cauchy sequence in \(X\) clusters;
    \item for each locally finite open cover of \(X\), there exists \(\mu > 0\) such that for each \(x \in X\), \(S^{\infty}_d(x, \mu)\) hits at most finitely many members of the cover;
    \item for each open cover \(\mathcal V\) of \(X\), there exists \(\mu > 0\) such that for each \(x \in X\), \(S^{\infty}_d(x, \mu)\) has a finite subcover from \(\mathcal V\);
    \item for each open cover \(\mathcal V\) of \(X\) directed by inclusion, there exists \(\mu > 0\) such that \(\{S^{\infty}_d(x, \mu) : x \in X\}\) refines \(\mathcal V\);
    \item whenever \(f : \langle X, d \rangle \to \langle Y, \rho \rangle\) is continuous, \(f\) is uniformly locally component bounded on \(X\);
    \item whenever \(f \in C(X, \mathbb{R})\), \(f\) is uniformly locally component bounded on \(X\);
    \item \(nslc(X)\) is compact, and \(\forall \delta > 0\), \(\exists \mu > 0\) such that \(x \in X \setminus S_d(nlc(X), \delta) \Rightarrow f_c(x) > \mu\);
    \item each sequence \(\langle x_n \rangle\) in \(X\) such that \(\lim_{n \to \infty} f_c(x_n) = 0\) clusters;
    \item whenever \(\rho\) is a metric equivalent to \(d\) on \(X\), there exists \(\delta > 0\) such that \(\forall x \in X\), \(S^{\infty}_d(x, \delta)\) is \(\rho\)-bounded;
    \item whenever \(\langle A_n \rangle\) is a decreasing sequence of non-empty closed subsets of $X$ with \(\lim_{n \to \infty} \inf\{ f_c(x) : x \in A_n \} = 0\), then \(\bigcap_{n=1}^{\infty} A_n \neq \emptyset\);
    \item whenever \(\langle A_n \rangle\) is a decreasing sequence of non-empty closed subsets of $X$ with \(\lim_{n \to \infty} \sup\{ f_c(x) : x \in A_n \} = 0\), then \(\bigcap_{n=1}^{\infty} A_n \neq \emptyset\).
\end{enumerate}
\end{theorem}

\begin{remark}
By condition (8), if \( \operatorname{nslc}(X) = \emptyset \) in a uniformly star superparacompact space, then \( \inf\{ f_c(x) : x \in X \}>0 \). As a result, a strongly locally compact space is uniformly star superparacompact if and only if it is strongly uniformly locally compact. We know that normed linear space is uniformly paracompact if and only if it is finite dimensional. However, non-trivial normed linear spaces are not uniformly star superparacompact since $\operatorname{nslc}(X)=X$.

    
\end{remark}
 By condition (2), we conclude that.
\begin{corollary}
Each closed metric subspace of a uniformly star superparacompact metric space is uniformly star superparacompact.
\end{corollary}

\begin{theorem}
Let $\langle X, d_1 \rangle$ and $\langle Y, d_2 \rangle$ be metric spaces, and equip $X \times Y$ with the box metric $\rho$. Then the product is uniformly star superparacompact if and only if one of the following conditions holds:
\begin{enumerate}
    \item either $\langle X, d_1 \rangle$ or $\langle Y, d_2 \rangle$ is compact and the other is uniformly star superparacompact, or
    \item both $\langle X, d_1 \rangle$ and $\langle Y, d_2 \rangle$ are strongly locally compact and uniformly star superparacompact.
\end{enumerate}
\end{theorem}
\begin{proof}
We write the measure of strong local compactness functionals for the factors by $f_c^X$ and $f_c^Y$ and write $f_c$ for the product. For sufficiency, suppose first that one of the spaces, say $\langle X, d_1 \rangle$, is compact while $\langle Y, d_2 \rangle$ is just uniformly star superparacompact. Let $(x, y) \in X \times Y$. If $S^{\infty}_{d_2}(y, \varepsilon)$ is compact, then $S^{\infty}_\rho((x, y), \varepsilon)$ is compact as $S^{\infty}_{d_1}(x, \varepsilon)$ is always compact. Thus, if $\langle (x_n, y_n) \rangle$ is a sequence in $X \times Y$ with $\lim_{n \to \infty} f_c((x_n, y_n)) = 0,$ then $\lim_{n \to \infty} f_c^Y(y_n) = 0$. As a result, $\langle y_n \rangle$ has a convergent subsequence. Also, $\langle x_n \rangle$ also has a convergent subsequence. Thus, by passing to subsequence, we conclude that $\langle (x_n, y_n) \rangle$ clusters.

Now suppose that $X$ and $Y$ are both strongly locally compact and uniformly star superparacompact. Since $\operatorname{nslc}(X) = \emptyset = \operatorname{nslc}(Y)$, we conclude by condition (8) of Theorem \ref{equivalentconditionsforspace} that $\inf \{ f_c^X(x) : x \in X \} > 0$ and $\inf \{ f_c^Y(y) : y \in Y \} > 0$. As a result, $\inf \{ f_c((x, y)) : (x, y) \in X \times Y \}>0$, so the product is uniformly star superparacompact by the same condition.

For necessity, assume $X \times Y$ is uniformly star superparacompact. From the immediately preceding corollary, $X$ and $Y$ both are uniformly star superparacompact. Suppose neither factor is compact and at least one, say $\langle X, d_1 \rangle$, is not locally compact, that is, $\operatorname{nslc}(X) \neq \emptyset$. Pick $x_0 \in\operatorname{nslc}(X)$ and a sequence $\langle y_n \rangle$ in $Y$ that does not cluster. Since $X \times Y$ is equipped with the box-metric, for each $n \in \mathbb{N}$ we have $(x_0, y_n) \in \operatorname{nslc}(X \times Y)$, but the sequence $\langle (x_0, y_n) \rangle$ fails to cluster in $X \times Y$. By condition (9) of Theorem \ref{equivalentconditionsforspace}, the product is not uniformly star superparacompact, which is a contradiction.
\end{proof}

We now specify what precisely must be added to uniform paracompactness to achieve uniform star superparacompactness for a metric space.

\begin{theorem}
Let $\langle X, d \rangle$ be a uniformly paracompact metric space. Then the space is a unifromly star superparacompact if and only if each sequence $\langle x_n \rangle$ in $X$ with $\lim_{n \to \infty} f_c(x_n) = 0$ has a Cauchy subsequence.
\end{theorem}

\begin{proof}
Suppose $\langle X, d \rangle$ is a  unifromly star superparacompact space, and let $\langle x_n \rangle$ satisfy $\lim_{n \to \infty} f_c(x_n) = 0$. By condition (9) of Theorem \ref{equivalentconditionsforspace}, $\langle x_n \rangle$ clusters, which means that it has a convergent and therefore a Cauchy subsequence.

For the converse, it is suffient to prove condition (2) of Theorem \ref{equivalentconditionsforspace}. Let $\langle x_n \rangle$ be a cofinally Boubaki quasi-Cauchy subsequence in $X$. Then for each $n \in \mathbb{N}$, let $M_n$ be an infinite subset of $\mathbb{N}$ such that $\forall j\in M_n$, $x_j\in S_d^{\infty}(p_n, \frac 1 n)$ for some $p_n\in X$.
If for some $n \in \mathbb{N}$, the subsequence $\langle x_j \rangle_{j \in M_n}$ clusters, we are done. Otherwise, $\langle x_n \rangle$ has a subsequence along which $f_c$ tends to zero which by assumption has a Cauchy subsequence. Since this subsequence is cofinally Cauchy and $X$ is a uniformly paracompact space, the subsequence clusters. Thus $\langle x_n \rangle$ clusters.
\end{proof}

\begin{theorem}\label{up_uss}
Let $\langle X, d \rangle$ be a uniformly paracompact metric space. Then the space is uniform star superparacompact if and only if $\forall \varepsilon>0$, $\exists\delta>0$ such that for every $x\in X$, we can find $x_1, x_2,\ldots x_k\in X$ satisfying $S_d^{\infty}(x, \delta)\subseteq \bigcup_{i=1}^k S_d(x_i, \varepsilon)$.
\end{theorem}

\begin{proof}
Suppose $\langle X, d \rangle$ is a uniform star superparacompact space. Let $\varepsilon>0$ be given. Consider the open cover $\{S_d(x, \varepsilon):x\in X\}$ of $X$. By condition (4) of Theorem \ref{equivalentconditionsforspace}, there exists \(\delta > 0\) such that for each \(x \in X\), \(S^{\infty}_d(x, \delta)\) has a finite subcover from the collection  $\{S_d(x, \varepsilon):x\in X\}$, and this proves the condition. 





For the converse, let \(\mathcal V\) be an open cover of \(X\). Since $X$ is uniformly paracompact, by condition (4) of \cite[Theorem 30.1]{BeerBook} there exists \(\mu > 0\) such that for each \(x \in X\), \(S_d(x, \mu)\) has a finite subcover from \(\mathcal V\). Hence by the given assumption, \(S^{\infty}_d(x, \mu)\) has a finite subcover from \(\mathcal V\). Hence $X$ is uniformly star superparacompact in view of condition (4) of Theorem \ref{equivalentconditionsforspace}.

\end{proof}

\begin{proposition}\label{cbc_uss}
    Let $\langle X, d \rangle$ be a cofinally Bourbaki-complete metric space. Then the space is uniform star superparacompact if and only if $\forall \varepsilon>0$, $\exists\delta>0$ such that for every $x\in X$, we can find $x_1, x_2,\ldots x_k\in X$ and $m\in\mathbb N$ satisfying $S_d^{\infty}(x, \delta)\subseteq \bigcup_{i=1}^k S^{m}_d(x_i, \varepsilon)$. 
\end{proposition}
   
\begin{proof}
 For the necessity, take $m=1$. The converse is an immediate consequence of the above theorem and \cite[Theorem 1.3.27.]{Thesis}  
\end{proof}

Before we move further into our analysis, we define the functional $\overline{f_c}:\mathcal{P}_0(X) \to [0, \infty]$ by \(
\overline{f_c}(A) = \operatorname{sup}\{f_c(a):a\in A\} .
\) 

The following functionals are recalled from \cite{Garrido2014, Thesis}; additional information and a more comprehensive discussion may be found therein.

Let \( P_0(X) \) be the family of non-empty subsets of \( X \). 

The functional \( \gamma : P_0(X) \to [0, \infty] \) is defined by
\[
\gamma(A) = \inf \left\{ \varepsilon > 0 : A \subseteq S^m_d(x, \varepsilon) \text{ for some } m \in \mathbb{N} \text{ and } x \in X \right\}.
\]
The functional \( \eta : P_0(X) \to [0, \infty] \) be defined by
\[
\eta(A) = \inf \left\{ \varepsilon > 0 : A \subseteq \bigcup_{i=1}^{k} S^m_d(x_i, \varepsilon) \text{ for some } m \in \mathbb{N}, \text{ and finite } x_i \in X, i = 1, \dots, k \right\}.
\]

The functional \( \alpha : P_0(X) \to [0, \infty] \) be defined by
\[
\alpha(A) = \inf \left\{ \varepsilon > 0 : A \subseteq \bigcup_{i=1}^{k} S_d(x_i, \varepsilon) \text{ for some } m \in \mathbb{N}, \text{ and finite } x_i \in X, i = 1, \dots, k \right\}.
\]
\begin{theorem}
Let \(\langle X, d \rangle\) be a complete metric space. Then the space is unfirom star superparacompact if and only if or every $\varepsilon > 0$, there exists $\delta > 0$ such that for all non-empty closed set $A$, if $\overline{f_c}(A)<\delta$, then $\alpha(A) < \varepsilon$.
\end{theorem}

The proof of the above theorem follows from \cite[Theorem 28.24]{BeerBook}. We now see what must be added to Boubaki-completeness of a metric space to produce uniform star superparacompactness for a metric space. But first we prove the following lemma, which tells that $\eta$ is continuous with respect to the Hausdorff distance.

\begin{lemma}\label{continuousfunctionallemma}
 Let $\langle X, d \rangle$ be a metric space, and suppose $A, A_1, A_2, A_3, \dots$ is a sequence of non-empty subsets of $X$ such that 
\(
\lim_{n \to \infty} H_d(A_n, A) = 0.
\)
Then
\[
\lim_{n \to \infty} \eta(A_n)= \eta(A).
\]  
\end{lemma}

\begin{proof}
 Let $\varepsilon>0$ be given. Then $\exists n_0\in\mathbb N$ such that $ A_n\subseteq A^{\frac{\varepsilon}{2}}$ and $ A\subseteq A_n^{\frac{\varepsilon}{2}}$ for all $n\geq n_0$. Also, there exist $m\in \mathbb N$ and $x_1,\ldots x_k\in X$ such that $A\subseteq \bigcup_{i=1}^k S_d^m(x_i, \eta(A)+\frac{\varepsilon}{2})$. Hence $A_n\subseteq \bigcup_{i=1}^k S_d^{m+1}(x_i, \eta(A)+\frac{\varepsilon}{2})$ for all $n\geq n_0$. Thus $\eta(A_n)\leq\eta(A)+\frac{\varepsilon}{2}$ for all $n\geq n_0$. Similarly, $\eta(A)\leq\eta(A_n)+\frac{\varepsilon}{2}$. As a result, $|\eta(A_n)-\eta(A)|<\varepsilon$ for all $n\geq n_0$. Hence, \(
\lim_{n \to \infty} \eta(A_n) = \eta(A).
\) 
\end{proof}

\begin{theorem}\label{BC-complete_to_USS}
Let $\langle X, d \rangle$ be a Bourbaki-complete metric space. Then the space is unfirom star superparacompact if and only if or every $\varepsilon > 0$, there exists $\delta > 0$ such that for all non-empty closed set $A$, if $\overline{f_c}(A)<\delta$, then $\eta(A) < \varepsilon$.
\end{theorem}

\begin{proof}
For sufficiency, let $\langle A_n \rangle$ be a decreasing sequence non-empty closed subset with $\lim_{n\to\infty} \overline{f_c}(A_n) = 0$. Let $\varepsilon > 0$ be given. By our assumption, there exists $\delta > 0$ such that $\overline{f_c}(A_n) <\delta$, then $\eta(A_n) < \varepsilon$, for all $n\in\mathbb N$. As a result, $\lim_{n\to\infty} \eta(A_n) = 0$. Since $\langle X, d \rangle$ is Bourbaki-complete, by \cite[Theorem 22]{Garrido2014} $\bigcap_{n=1}^{\infty} C_n \neq \emptyset$ and compact. Hence, by \cite[Theorem 28.20]{BeerBook}, $X$ is uniformly star superparacompact.

Conversely, suppose to the contrary that there exists $\varepsilon > 0$ such that for each $n \in \mathbb{N}$, there exists a non-empty closed subset $A_n$ of $X$ with $\overline{f_c}(A_n) \leq \frac{1}{n}$ but $\eta(A_n) \geq \varepsilon$. Let \(
F_n := \{x \in X : f_c(x) \leq \tfrac{1}{n}\}
\) and put 
\(
F := \bigcap_{n=1}^{\infty} F_n,
\)
which, by \cite[Lemma 28.23]{BeerBook}, is nonempty and compact, and moreover, $\lim_{n\to\infty} H_d(F_n, F) = 0$. Since $A_n \subseteq F_n$, by the monotonicity of $\eta$, we have $\eta(F_n)\geq \varepsilon$. Since $\eta$ is continuous with respect to the Hausdorff distance, we have $\eta(F) \geq \varepsilon$. But $\eta(F) = 0$ since $F$ is qC-precompact (in fact, compact), which leads to a contradiction.
\end{proof}

Next, we give a similar type of theorem for Boubaki quasi-complete metric spaces. First, we introduce the following functionals.

 We define $\gamma^* : \mathcal{P}_0(X) \to [0, \infty]$ by
\[
\gamma^*(A) = \inf \left\{ \varepsilon > 0 : A \subseteq S_d^{\infty}(x, \varepsilon) \text{ for some } x \in X \right\}.
\]
The functional $\gamma^*$ satisfies the following properties:
\begin{itemize}
    \item[(i)] If $A \subseteq B$, then $\gamma^*(A) \leq \gamma^*(B)$.
    \item[(ii)] $\gamma^*(A) = \gamma^*(\cl{A})$. In fact, for any $\varepsilon > 0$ such that $A \subseteq S_d^{\infty}(x, \varepsilon) $ for some $x \in X$, we have $\cl{A} \subseteq S_d^{\infty}(x, \varepsilon)$.
    \item[(iii)] $\gamma^*(A) = 0$ if and only if $A$ is a chainable subset of $X$.
\end{itemize}

Now, define another functional $\eta^* : \mathcal{P}_0(X) \to [0, \infty]$ by
\[
\eta^*(A) = \inf \left\{ \varepsilon > 0 : A \subseteq \bigcup_{i=1}^{k} S_d^{\infty}(x_i, \varepsilon) \text{for some finite } \{x_1, \dots, x_k\} \subseteq X \right\}.
\]
The functional \(\eta^*\) satisfies the following properties:
\begin{itemize}
    \item[(i)] If \(A \subseteq B\), then \(\eta^*(A) \leq \eta^*(B)\).
    \item[(ii)] \(\eta^*(A) = \eta^*(\cl{A})\).
    \item[(iii)] \(\eta^*(A) = 0\) if and only if \(A\) is a qC-precompact subset of \(X\).
    \item[(iv)] \(\eta^*(A \cup B) = \max\{\eta^*(A), \eta^*(B)\}\).
\end{itemize}

Thus, the next theorem characterizes Bourbaki quasi-complete metric spaces in terms of the above functionals.

\begin{theorem}\label{cantortype}
For a metric space \((X, d)\), the following statements are equivalent:
\begin{enumerate}
    \item \(X\) is Bourbaki quasi-complete.
    \item For every decreasing sequence \((A_n)_{n \in \mathbb{N}}\) of non-empty closed subsets of \(X\) which satisfies \(\lim_{n \to \infty} \eta^*(A_n) = 0\), the intersection \(A= \bigcap_{n \in \mathbb{N}} A_n\) is a non-empty compact set.
    \item For every decreasing sequence \((A_n)_{n \in \mathbb{N}}\) of non-empty closed subsets of \(X\) which satisfies \(\lim_{n \to \infty} \gamma^*(A_n) = 0\), the intersection \(A= \bigcap_{n \in \mathbb{N}} A_n\) is a non-empty compact set.
\end{enumerate}
\end{theorem}

\begin{proof}
(1) $\Rightarrow$ (2) For every \(n \in \mathbb{N}\), let \(x_n \in A_n\). Clearly, the set \(\{x_n : n \in \mathbb{N}\}\) is a qC-precompact subset of \(X\). By \cite[Theorem 2.2]{Adhikary2024, Das2021}, the sequence \((x_n)_{n \in \mathbb{N}}\) has a Bourbaki quasi–Cauchy subsequence which clusters by Bourbaki quasi-completeness. Then,
\[
A = \bigcap_{n \in \mathbb{N}} A_n \supset \bigcap_{n \in \mathbb{N}} \cl(\{x_m : m \geq n\}) \neq \emptyset.
\]
Since $\eta^*(A)\subseteq \eta^*(A_n)$ for each $n\in\mathbb N$, $\eta^*(A)=0$. Hence $A$ is qC-precompact. Since $A$ is closed and \(X\) is Bourbaki quasi-complete, $A$ is compact. 

(2) $\Rightarrow$ (3) Since \(\eta^*(A) \leq \gamma^*(A)\) for every \(A \subset X\), we have the desired consequence.

(3) $\Rightarrow$ (1) Let \((x_n)_{n \in \mathbb{N}}\) be a Bourbaki quasi–Cauchy sequence in \(X\). We define the decreasing sequence of non-empty closed sets \(
A_n = \cl \{x_m : m \geq n\}.
\)
Now, for each $\varepsilon>0$, there exists $n_0\in\mathbb N$ and $p\in X$ such that for all $m\geq n_0$, $x_m\in S_d^{\infty}(p, \frac{\varepsilon}{2})$. Hence $\gamma^*(A_n)<\varepsilon$ for all $n\geq n_0$. Hence \(\lim_{n \to \infty} \gamma^*(A_n) = 0\). Therefore, the set
\(
A = \bigcap_{n \in \mathbb{N}} A_n
\)
is non-empty and compact. Thus\((x_n)_{n \in \mathbb{N}}\) clusters.
\end{proof}

In a similar way to the proof of Theorem \ref{BC-complete_to_USS}, we can prove the following: 

\begin{theorem}\label{BCQ-complete_to_USS}
Let $\langle X, d \rangle$ be a Bourbaki quasi-complete metric space. Then the space is unfirom star superparacompact if and only if or every $\varepsilon > 0$, there exists $\delta > 0$ such that for all non-empty closed set $A$, if $\overline{f_c}(A)<\delta$, then $\eta^*(A) < \varepsilon$.
\end{theorem}

In terms of above functionals, we the following theorem are immediate from our above analysis.

\begin{theorem}\label{uniformly_char}
For a metric space $\langle X, d \rangle$ the following statements are equivalent:

\begin{enumerate}
    \item $X$ is uniformly star superparacompact;
    \item $X$ is uniformly paracompact and $\forall \varepsilon>0$, $\exists\delta>0$ such that for every non-empty subset $A$ of $X$ if $\gamma^*(A)<\delta$, then $\alpha(A)<\varepsilon$; 
    \item $X$ is cofinally Boubaki-complete and $\forall \varepsilon>0$, $\exists\delta>0$ such that for every non-empty subset $A$ of $X$ if $\gamma^*(A)<\delta$, then $\eta(A)<\varepsilon$; 
    \item $X$ is complete and $\forall \varepsilon>0$, $\exists\delta>0$ such that for every non-empty subset $A$ of $X$ if $\operatorname{min}\{\gamma^*(A), \overline{f_c}(A), \overline{\nu}(A)\}<\delta$, then $\alpha(A)<\varepsilon$; 
    \item either $X$ is uniformly locally compact, or $\operatorname{nlc}(X)$ is a non-empty compact set such that, for every $\varepsilon > 0$, the set $X \setminus \operatorname{nlc}(X)$ is uniformly locally compact in its relative topology, and there exists $\delta > 0$ such that
\(
\alpha(S^{\infty}_d(x, \delta)) < \varepsilon,
\)
whenever $x \in \operatorname{nlc}(X)$.     
\end{enumerate}
\end{theorem}
\begin{proof}
 We only need to prove \((2)~\Leftrightarrow~(5)\). Firstly, we assume that the condition (2) holds. Let $\varepsilon > 0$ be given. Then there exists $\delta > 0$ such that, for every $x \in X$, we have \(\alpha(S^{\infty}_d(x, \frac{\delta}{2})) < \varepsilon,\) because \(
\gamma^*(S^{\infty}_d(x, \frac{\delta}{2}))<\delta.\) The remainder of condition follows from the characterization of uniform paracompact space.

Conversely, since \(
\alpha(S^{\infty}_d(x, \delta)) < \varepsilon, 
\) we have \(
\alpha(S_d(x, \delta)) < \varepsilon
\). As a result, by \cite[Theorem 28]{Garrido2014} $X$ is uniformly paracompact and $\forall \varepsilon>0$, $\exists\delta>0$ such that for every non-empty subset $A$ of $X$ if $\gamma(A)<\delta$, then $\alpha(A)<\varepsilon$. Now, the condition follows from the fact that $\gamma(A)\leq \gamma^*(A)$ for each non-empty subset $A$ of $X$. 
  
\end{proof}

\noindent \textbf{Statement and Deceleration:} Some further investigations related to this work are not included in the present version. These include a characterization of the completion of uniformly star superparacompact spaces, as well as a study of such spaces in terms of Lipschitz-type functions. Based on a comparison of abstracts, the author assumes that \cite{Das2021} is the arXiv version of \cite{Adhikary2024}, as the full text of the latter was not accessible to the author.

\end{document}